\documentclass{amsart}
\usepackage{amssymb, epsfig}

\usepackage{amssymb, amsfonts, amsmath, amsthm, marvosym}
\usepackage{graphicx}
\usepackage{enumerate}
\usepackage{url} 
\usepackage{float}

\newcommand{\R}{\mathbf{R}}
\newcommand{\E}{\mathbf{E}}

\newcommand{\h}{\mathbf{H}}

\newcommand{\II}{\mathrm{I}\hspace{-0.8pt}\mathrm{I}}
\newcommand{\id}{\text{id}}

\newcommand{\M}{\E^{n+1}_1}

\newcommand{\dS}{\mathbf{dS}^{n+1}}

\newcommand{\adS}{\mathbf{adS}^{n+1}}

\theoremstyle{plain}
\newtheorem{thm}{Theorem}[section]
\newtheorem{cor}[thm]{Corollary}

\newtheorem{lem}[thm]{Lemma}
\newtheorem{definition}[thm]{Definition}

\theoremstyle{definition}

\newtheorem{remark}[thm]{Remark}

\DeclareMathOperator{\hess}{\overline{\nabla}^2\!}

\DeclareMathOperator{\Sec}{Sec}

\renewcommand{\tilde}{\widetilde}

\DeclareMathOperator{\length}{length}

\DeclareMathOperator{\sg}{sgn}

\numberwithin{equation}{section}
\begin{document}

\title{Space-time convex functions and sectional curvature}

\author{Stephanie B. Alexander}
\address{1409 W. Green St., Urbana, Illinois 61801}
\email{sba@illinois.edu}
\thanks{This work was partially supported by a grant from the Simons Foundation (\#209053 to Stephanie Alexander).}
\author{William A. Karr}
\address{1409 W. Green St., Urbana, Illinois 61801}
\email{wkarr2@illinois.edu}
\thanks{This material is partially based upon work supported by the National Science Foundation Graduate Research Fellowship to William Karr under Grant No. DGE 11-44245. Any opinion, findings, and conclusions or recommendations expressed in this material are those of the authors and do not necessarily reflect the views of the National Science Foundation.}

\begin{abstract}
We show that 
in Lorentzian manifolds, 
sectional curvature bounds  of the form $\mathcal{R}\le K\,$, as defined by Andersson and Howard, are closely tied to space-time convex and $\lambda$-convex ($\lambda>0$) functions, as defined by Gibbons and Ishibashi.  Among the consequences are a natural construction  of such functions, and an analogue, that applies to domains of a new type, of a theorem of Al\'ias, Bessa and deLira ruling out trapped submanifolds.
\end{abstract}
\maketitle 


\section{Introduction}\label{introduction}
A study of the possible uses of convex functions in General Relativity was initiated by Gibbons and Ishibashi, according to whom:
``Convexity and convex functions play an important role in theoretical physics $\ldots\ $ 
[and] also have important applications to geometry, including Riemannian geometry  
$\ldots\ $ 
It is surprising therefore that, to our knowledge, 
that techniques making use of convexity and convex functions have played no great role in General Relativity'' \cite{gi}.

Gibbons and Ishibashi introduce and mainly consider  ``space-time convex'' functions  on Lorentzian manifolds $(M,g)$, or  more generally,  functions $f$ satisfying
$$\hess f \ge \lambda\,g,\ \ \ \lambda >0.$$
They find  examples and non-examples of such functions  on regions in cosmological space-times and black-hole space-times.  They show,  
for example, that such functions rule out closed marginally  inner and outer trapped surfaces. Curvature bounds do not arise in their considerations.

The purpose of this note is to show that sectional curvature bounds  of the form $\mathcal{R}\le K\,$ are closely tied to space-time convex functions.  Among the consequences:
 \begin{itemize}
 \item
A natural construction  of such functions.
 \item
New domains that cannot support
trapped submanifolds, namely a full neighborhood of a point $\,q$, rather than a neighborhood of $q$ 
in the
chronological
 future of $q$ as has been considered previously, 
 in particular by Al\'ias, Bessa and deLira \cite{abl}.

\end{itemize}

The bound $\mathcal{R}\le K\,$,  introduced by Andersson and Howard \cite{ah}, extends $\Sec\le K$ from the Riemannian to the semi-Riemannian setting by requiring spacelike sectional curvatures to be $\le K$ and timelike ones to be $\ge K$.   
Equivalently,
the curvature tensor is required to satisfy
$$
g(R(v,w)v,w)\, \le
\,K \bigl(g( v,v )\,g(w,w)  - g(v,w)^2\bigr).
$$
For $\mathcal{R}\ge K\,$, reverse the inequalities. 

In addition, we 
indicate connections between investigations that have been  pursued independently by various authors, including: 
\begin{itemize}
 \item
Comparison theorems for Lorentzian distance on
domains in the
chronological
future of a source point or hypersurface on which the source has no Lorentzian cut points, given  timelike sectional curvature controls
(see for example  \cite{egk,ahp,imp,abl}). 

\item
Hessian comparisons   on level hypersurfaces in exponentially embedded  neighborhoods of a point or hypersurface, 
given a sectional curvature bound of the form 
$\mathcal{R}\le K\,$ or $\mathcal{R}\ge K\,$ 
\cite{ah,ab-lorentz}.  
\item
Space-time convex functions \cite{gi}. 
\end{itemize}

\subsection{Outline of the paper}
Section 2 is an introduction to space-time convex and $\lambda$-convex functions, as defined in  \cite{gi}. 

Section 3 summarizes certain theorems about Hessian and Laplacian comparisons on the Lorentzian distance function from a point or achronal spacelike hypersurface, under  comparisons on timelike sectional curvature (\cite{egk,ahp,imp,abl}).


Section 4 describes results 
from \cite{ah, ab-lorentz} 
concerning the conditions $\mathcal{R} \geq K$ and $\mathcal{R} \leq K$ in semi-Riemannian manifolds.
 In particular, in \cite{ah}  Andersson and Howard prove a comparison theorem for matrix Ricatti equations which applies to the second fundamental forms  of parallel families of hypersurfaces  under curvature comparisons.  In \cite{ ab-lorentz}, this theorem is adapted to tubes around points; as an application, the geometric meaning of the bounds $\mathcal{R} \geq K$ and $\mathcal{R} \leq K$ is found by introducing signed lengths of geodesics.

In section 5, we use this framework to rule out trapped submanifolds in an exponentially embedded neighborhood of a point in a space-time satisfying $\mathcal{R} \leq K$.


\section{Space-time convex functions}
\begin{definition}
\label{def:lambda-convex1}
Given smooth   functions $f:M\to\R$ and $\lambda:M\to\R$ on a semi-Riemannian manifold $(M,g)$, $f$ will be called \emph{$\lambda$-convex} if the Hessian $\hess f$ satisfies 
\begin{equation}\label{eq:space-time-convex1}
\hess f \ge \lambda\,g,
\end{equation} 
or equivalently, 
\begin{equation}\label{eq:space-time-convex2}
(f\circ\gamma)'' \ge (\lambda \circ \gamma)  \,g(\gamma',\gamma') 
 \end{equation}
 for every geodesic $\gamma$.
 
 Suppose $M$ is Lorentzian.  We say $f$ is \emph{space-time $\lambda$-convex} if  $f$ is $\lambda$-convex for some \emph{positive} function $\lambda$, and
 $\hess f$ has Lorentzian signature.
 \end{definition}
 
Note that this definition differs from the classical definition of convexity in that the right-hand sides of (\ref{eq:space-time-convex1}) and (\ref{eq:space-time-convex2}) need not be positive when $\lambda>0$.  Rather, controlled concavity is allowed along timelike geodesics, 
and is imposed in the definition of space-time convexity.
 
One of the simplest examples of a space-time $\lambda$-convex function is 
\begin{equation}
\label{eq:mod-dist-0}
f(\mathbf x, t)=\frac{1}{2}(\mathbf x\cdot\mathbf x-\lambda t^2),\quad (\mathbf x, t)\in\M,
\end{equation}
on Minkowski space for some constant $0<\lambda\le 1$.
 
 As pointed out in \cite{gi}, the geometric meaning of space-time convexity is that at each point, the forward light cone defined by the Hessian $\,\hess f\,$ lies
  inside the light cone defined by the space-time metric.

Definition \ref{def:lambda-convex1} is consistent with current Riemannian/Alexandrov usage of ``$\lambda$-convex'' (see \cite{p}); and 
also with the definition of ``space-time convex'' in \cite{gi} except that our $\lambda$ is a positive function and Gibbons and Ishibashi take $\lambda$ to be a positive constant.
(However, Definition \ref{def:lambda-convex1} differs from the usage in \cite{ab-lorentz}.)

In \cite{gi}, Gibbons and Ishibashi begin  an investigation  of the geometric implications  of space-time convex functions.  For example, they show that a space-time with a closed marginally inner and outer trapped surface  cannot support a space-time convex function.  
  
Here a \emph{marginally inner and outer trapped surface} $\Sigma$  is a 
spacelike submanifold of codimension 2 
whose mean curvature 
vanishes. 

Seeking
examples of space-time convex functions,  Gibbons and Ishibashi consider \emph{Robertson-Walker spaces} 
\[M=-I\times_fF,\]
that is, $M$ is the product manifold carrying the  warped product metric 
\[
-d\tau^2 +f^2
ds^2_F
\]
where $I=(a,b)$,\, $a\in [-\infty,\infty)$,\, $b\in (-\infty,\infty]$\,, $f:I\to \R_+$,
and  $F$ has constant sectional curvature. 
They ask when the function
\begin{equation}\label{eq:G-I-fn}
-f^2/2
 \end{equation}
is space-time convex (here we use $f$ to denote both the  warping function and its lift to $M$). For instance, various cosmological charts are considered on de-Sitter space $\dS$ and anti-de-Sitter space $\adS$. One of these yields an affirmative answer: namely, the function (\ref{eq:G-I-fn}) is space-time convex on the region 
$$(0,\pi/2)\times_{\sin}\h^n$$
in $\adS$. 

Gibbons and Ishibashi do not consider curvature bounds when seeking examples. The perspective of space-times with curvature bounds of the form $\,\mathcal{R}\le K\,$ suggests an alternative,  namely analogues of the ``square norm'' ((\ref{eq:mod-dist-0}) with $\lambda =1$).
For instance,
these analogues yield space-time convex functions adapted 
to 
some of the domains in de-Sitter and anti-de-Sitter space considered in \cite{gi}.  

Our theorems show that space-time convex functions arise naturally in all Lorentzian manifolds satisfying  $\mathcal{R}\le K$.

\section{Comparisons for Lorentzian distance}
\label{sec:lor-dist}
 
Let us mention some related works concerning the Lorentian distance functions from a point or spacelike hypersurface.  
All  these investigations are restricted to domains  containing no Lorentzian cut points of the source point or hypersurface.

\begin{enumerate}
\item
\label{egk}
In \cite{egk}, Erkekoglu, Garcia-Rio and Kupeli 
prove  Hessian and Laplacian comparison theorems  for level sets of the Lorentzian distance function from points or  from achronal spacelike hypersurfaces, in two space-times $M$ and $\tilde M$. They consider corresponding timelike,  distance-realizing unit geodesics in $M$ and $\tilde M$, where sectional curvatures of 2-planes tangent to  the geodesics at corresponding values of the time parameter are no greater in $M$ than in $\tilde M$. Some space-time singularity theorems are given.

\item
In \cite{ahp}, Al\'ias, Hurtado and Palmer study the restriction of Lorentzian distance from a point or spacelike hypersurface to a spacelike hypersurface satisfying the Omori-Yau maximum principle. 
  Under constant bounds either above or below on timelike sectional (or Ricci) curvatures, they obtain 
sharp estimates on the mean curvature of such hypersurfaces.

\item
In \cite{imp},
Impera studies Hessian and Laplacian comparisons  for  Lorentzian distance from a point, assuming timelike sectional curvatures are bounded above or below by a function of the Lorentzian distance.  Estimates are obtained on the higher order mean curvatures of spacelike hypersurfaces satisfying the Omori-Yau maximum principle.

\item
In \cite{abl}, Al\'ias, Bessa and deLira 
prove non-existence results and sharp mean curvature estimates for   trapped submanifolds (of arbitrary codimension), based on comparison inequalities for the Laplacian of the restriction to a spacelike submanifold  of the Lorentzian distance function from a point or achronal spacelike hypersurface.  They use a 
weak Omori-Yau maximum principle equivalent to stochastic completeness. 
\end{enumerate}

\section{Curvature bounds $\,\mathcal{R}\le K$, \,$\mathcal{R}\ge K$.}
Recall that  $\mathcal{R}\le K\,$  means that spacelike sectional curvatures are $\le K$ and timelike ones are $\ge K$. For $\mathcal{R}\ge K\,$, reverse the inequalities.  
(Note that $\mathcal{R}\le K\le K'$ does not imply   $\mathcal{R}\le K'\,$!\,)

\subsection{Geometric meaning}
Briefly, $\mathcal{R}\le K\,$ means, as in the Riemannian case, that unit geodesics radiating from a point ``repel'' each other at least as much as in a space of  constant curvature $K$, assuming the same initial conditions.  However, repulsion here is meant in the \emph{signed} sense.  In particular, in the Lorentzian case,  if the initial direction of variation of the geodesics is timelike, we see \emph{negative repulsion}, that is, 
at least as much \emph{attraction} as in a Lorentzian space of constant curvature $K$. This is explained below in Subsections \ref{subsec:point-comp} and \ref{subsec:characterize}.

\subsection{GRW spaces}

Space-times satisfying  $\mathcal{R}\le K$ and $\mathcal{R}\ge K$ are abundant.  We mention  as examples, \emph{generalized Robertson-Walker (\,GRW\,) spaces}, 
namely warped products 
$\,M=(-I)\times _f F\,$ 
for arbitrary Riemannian manifolds $\,F$.

\begin{lem}\cite{ab-lorentz}
\label{lem:R>K}
A GRW space
$ M=-I\times_fF$\, satisfies  \,$\mathcal{R}\le K$\, if and only if  \,$f:I\to \R_+$ is
 $(-Kf)$-convex, that is,
 \[
f''\ge -Kf,
\]
and \,$F$ either is $1$-dimensional or has sectional curvature \,$\le C$\, where
 \[
 C\ =  \ \inf\,( Kf^2-(f')^2).
\]
(For $\mathcal{R}\ge K$,  reverse the inequalities and substitute \,$\sup$\,  for \,$\inf$.)
\end{lem}
%

\subsection{Comparisons based at a point}
\label{subsec:point-comp}
Let $M$ be a semi-Riemannian manifold, and $U$ be the diffeomorphic image under $\exp_q$ of a  star-shaped region in $T_qM$ about $O$.  
Let $\gamma_{p,q}$  be the geodesic path in $U$ from $p$ to $q$ that is distinguished by this diffeomorphism.


Define the  \emph{signed energy function}
$\,E_q:U\to\R\,$ by 
\begin{equation}\label{eq:E}
E_q(p)= (\sg\gamma_{p,q})\,(\length\gamma_{p,q})^2,
\end{equation}
where $\sg\,\gamma$ take values $1,0,-1$ according to whether $\gamma_{p,q}$ is spacelike, null or timelike, respectively.  

Signing  was shown in \cite{ab-lorentz} to be the key to  geometric understanding of the curvature bounds $\mathcal{R}\le K$ and $\mathcal{R}\ge K$.  In particular, Ansersson and Howard do not consider signed distance or energy.

For a fixed choice of
$K\in\R$ and $q\in U$, define $\,f_{K,q}:U\to\R\,$  by
\begin{equation}\label{eq:modE}
f_{K,q} =\sum_{n=1}^\infty \frac{(-K)^{n-1}(E_q)^n}{(2n)!}
=\begin{cases}
E_q/2, & K=0,\\
(1 - \cos\sqrt{KE_q})/K , & K\ne 0.
\end{cases}
\end{equation}
Here the argument of $\cos$ may be imaginary, yielding $\cos it=\cosh t$.

\begin{remark}
\label{rem:mod-shape}
Note that on the lift of $\,U\,$ to $\,T_qM$ by  $(\exp_q)^{-1}\,$, the lift of $f_{K,q}$ is the square norm if $\,K=0$,  and an analogue if $\,K\ne 0 $. 
The possible values of $(1-Kf_{K,q})$ are $1$, $\cos\sqrt{|KE_q|}$ and $\cosh\sqrt{|KE_q|}$.
\end{remark}

Set
$f=f_{K,q}$ as in (\ref{eq:modE}), for a fixed choice of
$K$ and $q$.  Define the
\emph{modified shape operator} 
$S=S_{K,q}$ 
to be the self-adjoint operator associated with the
Hessian of $f$, namely,
\begin{equation}
\label{eq:Sdef1}
Sv=\overline{\nabla} _v \overline{\nabla} f
\end{equation}
where $\overline{\nabla}$ is the covariant derivative of $M$.

Note that the levels of  $f$ are the levels of $E_q$. The  
form of  $f$ was
chosen for analytic convenience 
(following \cite{Kr}), so that if  $M$ has constant curvature $K$ then $S$ is
a scalar multiple of the identity, namely
$S=(1-Kf)\,I$.

The modified shape operator $S$ has
the following further properties:  along a nonnull geodesic
from $q$, its restriction to normal vectors is a scalar multiple of
the second fundamental form of the level  hypersurfaces of
$E_q$; it is smoothly defined on the regular set of $E_q$, hence along
null geodesics from $q$ (as the second fundamental forms are not);
and finally, it satisfies a matrix Riccati equation along every
geodesic from $q$, after reparametrization as an integral curve of $\overline{\nabla}
f_{K,q}$.

The proof of the following theorem is by adapting to the set-up just described, a comparison theorem of Andersson and Howard  \cite[Theorem 3.2]{ah} that applies to exponentially embedded tubes about  hypersurfaces rather than points (see Subsection \ref{subsec:ah}). 
%
%

We say two geodesic segments $\sigma$ and $\tilde{\sigma}$ in
semi-Riemannian manifolds $(M,g)$ and $(\tilde{M},\tilde g)$ \emph{correspond} if
they are defined on the same affine parameter interval  and satisfy
$g(\sigma',\sigma')
=\tilde g\,(\tilde{\sigma}',\tilde{\sigma}')$. Let $R_{\sigma'}$ be the self-adjoint  operator
$R_{\sigma'}v = R(\sigma',v)\sigma'$, and similarly for $\tilde{R}_{\tilde\sigma'}$.

In the special case that the geodesics $\sigma$ and $\tilde\sigma$ are timelike, the following theorem includes    comparison inequalities of Erkekoglu, Garcia-Rio and Kupeli \cite[Theorem 3.1]{egk} for level hypersurfaces of the Lorentzian distance from a point.
However, here we are analyzing an  exponentially embedded \emph{neighborhood of a point} rather than restricting to the 
chronological future.

\begin{thm}\cite{ab-lorentz}
\label{thm:Scompare1}
%
%
Let $M$ and $\tilde{M}$ be semi-Riemannian manifolds of the same
dimension and index.  
For $q\in M$ and $\tilde q\in\tilde M$, let $U$ and $\tilde U$ be 
diffeomorphic images under $\exp_q$ and $\exp_{\tilde q}$  respectively of 
star-shaped regions about the origin in $T_qM$ and $T_{\tilde q}\tilde M$.
  Let $\sigma$ and $\tilde{\sigma}$ be
corresponding non-null geodesics  in $U$ and $\tilde U$ respectively, radiating from 
 $q$ and $\tilde q$.

Identify linear operators on $T_{\sigma(t)}M$ with
those on  $T_{\tilde{\sigma}(t)}\tilde{M}$ by parallel translation to the
basepoints, together with an isometry of $T_qM$ and
$T_{\tilde{q}}\tilde{M}$ that identifies $\sigma'(0)$ and
$\tilde{\sigma}'(0)$.  

Suppose $R_{\sigma'}\le\tilde{R}_{\tilde{\sigma}'}$
at corresponding points of $\sigma$ and  $\tilde{\sigma}$.  Then the modified
shape operators $S=S_{K,q}$ and $\tilde S=\tilde S_{K,q}$, as in (\ref{eq:Sdef1}),
satisfy $S\ge\tilde{S}$ (that is, $S-\tilde{S}$ is positive semidefinite) at corresponding points of
$\sigma$ and  $\tilde{\sigma}$.
\end{thm}

\begin{remark}
\label{rem:cone}
A more precise statement of Theorem \ref{thm:Scompare1} localizes at a choice of  unit geodesics $\sigma:[0,a]\to M$ and $\tilde\sigma:[0,a]\to\tilde M$, where $\sigma$ and $\tilde\sigma$ have no conjugate points.  Specifically, we let
$U\subset M$ and $\tilde U\subset\tilde M$  be  
diffeomorphic images under $\exp_q$ and $\exp_{\tilde q}$  of 
truncated cones of the form 
$(0,a]\times_{\id}D$ and $(0,a]\times_{\id}\tilde D$
with vertices at the origin, 
where $D$  and $\tilde D$ are open disks in the unit tangent ``spheres'' at $q$ and $\tilde q$ centered at $\sigma'(0)$ and $\tilde\sigma'(0)$ respectively. 
\end{remark}

%

The following basic lemma is verified in \cite{ab-lorentz}:

\begin{lem}
\label{lem:constant}
Let $M$ be a semi-Riemannian space of constant curvature $K$,  and $U$ be the diffeomorphic image under $\exp_q$ of a  star-shaped region in $T_qM$ about $O$.   Then $f\,_{K,q}:U\to\R$ satisfies
$$
\hess f\,_{K,q} = (1-Kf_{K,q})\,g.
$$
\end{lem}


Combining Theorem \ref{thm:Scompare1} and Lemma \ref{lem:constant}, we obtain: 

\begin{thm}\cite{ab-lorentz}
\label{thm:ncp-compare}
Let $M$ be a semi-Riemannian manifold satisfying $\mathcal{R}\le K$.  Let $U$ be the diffeomorphic image under $\exp_q$ of a  star-shaped region in $T_qM$ about $O$. Assume $\,E_q:U\to\R\,$ satisfies 
$\,E_q < \pi^2/K\,$ if $\,K> 0$, and 
$\,E_q > \pi^2/K\,$ if $\,K<0$.
Then $f\,_{K,q}:U\to\R$ satisfies 
$$
\hess f\,_{K,q} \ge (1-Kf_{K,q})\,g.
$$ 
That is, $f_{K,q}$ is $\,(1-Kf_{K,q})$-convex.
\end{thm}

\subsection{Geometric characterization of $\,\mathcal{R}\le K$,\, $\mathcal{R}\ge K$}
\label{subsec:characterize} 
The geometric characterization of Riemannian  sectional curvature bounds $\,\Sec\le K$ or $\,\Sec\ge K$ is given by local triangle comparisons with  Riemannian space forms of constant curvature $K$.  This is the basis of Alexandrov geometry, which extends the theory of Riemannian manifolds with sectional curvature bounds to highly singular spaces.

It turns out that this characterization by local triangle comparisons extends to semi-Riemannian manifolds \emph{if we take lengths of geodesics to be signed}. 

Recall that in a semi-Riemannian manifold, any point $q$ has arbitrarily small 
normal  neighborhoods $U$,  that is, $U$ is the diffeomorphic exponential image of a star-shaped 
domain in the tangent space of each of its points. There is a unique
geodesic $\gamma_{p,q}$ in $U$ between any two points  $p,q\in U$. 


\begin{thm}[\cite{ab-lorentz}]
\label{thm:ab-comparison1}
Let $M$ be a semi-Riemannian manifold.
\begin{enumerate}
\item
\label{geometric}
 If $M$ satisfies $R \le K \,\,(R\ge K)$, and
$U$ is a normal neighborhood for $K$, then the signed length of
the geodesic
between two points on any geodesic triangle of  $U$ is at most \,(at least)\,
that for the corresponding points on a \emph{model} triangle with the same signed sidelengths in a semi-Riemannian model surface $M_K$ with constant sectional curvature $K$. (For a nondegenerate triangle, $M_K$ is uniquely determined, as is the comparison model triangle up to motion.)

\item
Conversely, if these triangle comparisons hold in some normal neighborhood
of each point of $M$, then $R \le K \,\,(R\ge K)$.
\end{enumerate}
\end{thm}

\begin{remark} In \cite{harris} (see also \cite{harris-app}), Harris proves \emph{global} purely timelike triangle comparisons in space-times of timelike sectional curvature bounded above.  Thus the theorem of Harris is  a timelike version for Lorentzian manifolds of Toponogov's Globalization Theorem for Riemannian manifolds of sectional curvature bounded below \cite{toponogov}. 
\end{remark}
 
%

\subsection{Comparisons for parallel families of hypersurfaces}
\label{subsec:ah}
In \cite[Theorem 3.2]{ah}, Andersson and Howard prove a comparison theorem for matrix Riccati equations that applies to the second fundamental forms of parallel families of hypersurfaces of any signature in semi-Riemannian manifolds, rather than only to parallel families of spacelike hypersurfaces in Lorentzian manifolds as in Section \ref{sec:lor-dist}.  We give an analogue in Theorem \ref {thm:Scompare1}. 
  
For $\mathcal{R}\le 0$ and $\mathcal{R}\ge 0$, Andersson and Howard prove ``gap'' rigidity theorems  of the type first proved for Riemannian manifolds with $\Sec\le 0$ by Gromov \cite{bgs}, and with $\Sec\ge 0$ by Greene and Wu \cite{gw}, respectively.  As applications, they obtain  rigidity results for semi-Riemannian manifolds with simply connected ends of constant curvature.
 
We remark that while in the Riemannian case,  the Ricatti  comparisons of  \cite{ah} reduce to 1-dimensional equations (see
\cite{Kr}) the semi-Riemannian case seems to require matrix-valued  
equations.  Such increased complexity is perhaps not surprising, since semi-Riemannian curvature bounds
above (say) share some behavior with 
Riemannian  curvature bounds below as well as above.

\section{Results}
By Theorem \ref{thm:ncp-compare} we have:
\begin{cor}\label{cor:lam-convex}
Let $M$ be a semi-Riemannian manifold satisfying $\mathcal{R}\le K$.  Let $U$ be the diffeomorphic image under $\exp_q$ of a  star-shaped region in $T_qM$ about $O$. Assume $\,E_q:U\to\R\,$ satisfies
$\,E_q < \pi^2/4K\,$ if $\,K> 0$, and 
$\,E_q > \pi^2/4K\,$ if $\,K<0$.  Then $f\,_{K,q}:U\to\R$ is $\,\lambda$-convex with 
 $\lambda = 1-Kf_{K,q}>0$
 (where $f\,_{K,q}$ is defined in  (\ref{eq:E})  and (\ref{eq:modE})) .

Moreover, $f\,_{K,q}$ is space-time convex on a neighborhood of  $q$.
\end{cor}

\begin{proof}
By Theorem \ref{thm:ncp-compare}, $f\,_{K,q}:U\to\R$ is $\,(1-Kf_{K,q})$-convex. 
By (\ref{eq:modE}), setting $\lambda = 1-Kf_{K,q}$, we have
\begin{equation}\label{eq:m1-f}
\lambda\ =\ \begin{cases}
1, & K=0,\\
\cos\sqrt{KE_q} , & K\ne 0.
\end{cases}
\end{equation}

Suppose $K>0$.   If $E_q\le 0$, then $\lambda=\cosh\sqrt{|KE_q|} >0$.  If $\,0\le E_q < \pi^2/4K\,$, then 
$\lambda=\cos\sqrt{|KE_q|} >0\,$.  
Similarly for $K<0$.  

It remains to show $\hess f_{K,q}$ has Lorentzian signature in a neighborhood of $q$. This follows by continuity, since for a unit timelike geodesic $\gamma$ satisfying $\gamma(0)=q$ we have  
 $(f_{K,q}\circ\gamma)''(0)=-1$.
\end{proof}




In defining the second fundamental form $\II$  and mean curvature vector field $H$ 
of a $k$-dimensional submanifold $\Sigma$ of a Lorentzian manifold $M$, we use the convention in relativity (the opposite of that in differential geometry):
\begin{equation}
\label{eq:II-def}
\overline{\nabla}_X Y = \nabla_X Y - \II(X,Y),
\end{equation}
\begin{equation}
\label{eq:mean-curv}
H = \frac{1}{k} \sum_{i} \II(E_i,E_i),
\end{equation}
where $\overline{\nabla}$ and $\nabla$ denote the covariant derivatives on $M$ and $\Sigma$ respectively, and $\{ E_1, ... , E_k \}$ is a local orthonormal frame on $\Sigma$.
 
We are going to follow \cite{abl} in considering submanifolds $\Sigma$  satisfying the \emph{weak maximum principle} of Pigola, Rigoli and Setti \cite{prs1}, according to which for any smooth function $u$ on $\Sigma$ with $u^* = \sup_{\Sigma} u < +\infty$, there exists a sequence of points $p_n \in \Sigma$ such that $$
u(p_n) > u^* - \frac{1}{n} \quad \text{and} \quad \Delta u(p_n) < \frac{1}{n}.
$$ 
Pigola, Rigoli and Setti proved that $\Sigma$ satisfies the weak maximum principle if and only if $\Sigma$ has the probabilistic property of stochastic completeness \cite{prs1,prs2}.

 By \cite[Proposition 8]{gi}, domains carrying space-time convex functions $f$ cannot contain closed marginally inner and outer trapped surfaces.   
The proof extends to  the following proposition,
which does not depend on the behavior of $\hess f$ on causal vectors or on the codimension, and uses the weak maximum principal  to extend from closed to stochastically complete submanifolds.

\begin{thm}
\label{thm:miots local} 
Let $M$ be a Lorentzian manifold and $f : M \to \R$ be 
$\lambda$-convex on spacelike vectors
for some function $\lambda: M \to \R$. 
Then:
\begin{enumerate}
\item[(i)]
$M$ contains no stochastically complete spacelike submanifold with vanishing mean curvature 
and on which $f$ is bounded above 
and 
$\lambda$ has positive infimum.
\item[(ii)]
If $\lambda >0$, 
then  $M$ contains no closed spacelike submanifold with vanishing mean curvature.
\end{enumerate}
\end{thm}

\begin{proof}
Suppose $\Sigma$ is a spacelike $k$-dimensional submanifold with vanishing mean curvature. 
Let $\overline{\nabla}$ and $\nabla$ denote the covariant derivatives on $M$ and $\Sigma$ respectively. Let $\II$ and $H$ denote the  second fundamental form and mean curvature vector field of $\Sigma$ respectively. Let $u = f|_{\Sigma} : \Sigma \to \R$ denote the restriction of $f$ to $\Sigma$.

Then for any $x \in T_p \Sigma$, 
$$
\label{eq:Hessian-of-restriction}
(\nabla^2 u)_p (x,x) = (\overline{\nabla}^2 f)_p(x,x) - g( \II_p(x,x), \overline{\nabla} f_p ).
$$
 If $\{ e_i \}$ is an orthonormal basis for $T_p \Sigma$, then 
\begin{equation}
\label{eq:laplacian-of-restriction}
\Delta u (p) = \sum_{i = 1}^{k} (\overline{\nabla}^2 f)_p(e_i,e_i) - k \, g(H_p, {\overline{\nabla} f}_p).
\end{equation} Since $f$ is $\lambda$-convex and $H$ vanishes, $u$ satisfies
$$
\label{eq:laplacian-of-restriction1}
\Delta u \geq k \, \lambda|_\Sigma.
$$

Thus if the Laplacian $\Delta u$ is bounded below by $ k \, \inf_{\Sigma} \lambda > 0 $, 
and $u$ is bounded above, then $\Sigma$ cannot be stochastically complete.  
This proves (i), and (ii) follows.
\end{proof}

\begin{definition}
In a causally orientable Lorentzian manifold, a spacelike submanifold $M$ whose mean curvature vector field is causal and future-pointing is called  a \emph{weakly future-trapped submanifold}.
\end{definition}

\begin{remark}
Galloway and Senovilla prove that standard singularity theorems hold in Lorentzian manifolds of arbitrary dimension with closed trapped 
submanifolds of arbitrary co-dimension  \cite{gs10}. They point out
that such submanifolds appear to have many common properties independent of the codimension.
\end{remark}

%
%
%
%
%
%
%
%

The significance of the following theorem lies in using sectional curvature bounds to examine geometric properties of a full neighborhood of a point $\,q$, rather than restricting 
to the
chronological
 future of $q$.

If in the following theorem we restrict $U$ and $\tilde U$ to the 
chronological
future of $q$ and assume only timelike sectional curvature $\ge K$, then taking into account Remark \ref{rem:cone}, we obtain 
a result of Al\'ias, Bessa and deLira  
(\cite[Corollary 4.2]{abl}).

\begin{thm}
\label{thm:trapped-SC}
Let $M$ be a Lorentzian manifold satisfying $\mathcal{R}\le K$.  Let $U$ be a domain in $M$ that is the diffeomorphic image under $\exp_q$ of a  star-shaped region in $T_qM$ about $O$. Suppose that $\,E_q:U\to\R\,$ is bounded above and satisfies $\,E_q < \pi^2/4K\,$ if $K > 0$ and $\,E_q > \pi^2/4K\,$ if $K < 0$.
\begin{enumerate}
\item[(i)]
Then $U$ contains no stochastically complete spacelike submanifolds $\Sigma$ with vanishing mean curvature, and such that 
$\,\sup E_q|_\Sigma < \pi^2/4K\,$ if $\,K> 0$ and $\inf E_q|_\Sigma > \pi^2/4K\,$ if $\,K<0$.
\item[(ii)]
More generally, $U$ contains no stochastically complete, weakly future-trapped submanifold whose mean curvature vector field $H$ satisfies 
\begin{equation}\label{eq:HE_q}
H E_q \leq 0,
\end{equation}
and such that 
$\,\sup E_q|_\Sigma < \pi^2/4K\,$ if $\,K> 0$ and $\inf E_q|_\Sigma > \pi^2/4K\,$ if $\,K<0$.

\item[(iii)]
Suppose 
$K\ne 0$ and 
$U\subset \tilde{U}$,  where $\tilde{U}$ is the diffeomorphic image under $\exp_q$ of a  star-shaped region in $T_qM$ about $O$, and $\,E_q:\tilde{U}\to\R\,$  satisfies $\,E_q < \pi^2/K\,$ if $\,K> 0$ and 
$\,E_q > \pi^2/K\,$if $\,K<0$.
Then no stochastically complete, weakly future-trapped submanifold in $\tilde{U}$ that satisfies $H E_q \leq 0$  enters $U$.
\end{enumerate}
\end{thm}
\begin{proof}
By Corollary \ref{cor:lam-convex}, the function $f\,_{K,q}:U\to\R$ as defined in (\ref{eq:E}) and (\ref{eq:modE}) is $\lambda$-convex with $\lambda = 1 - K f_{K,q} > 0$. Suppose $\Sigma$ is a weakly future-trapped $k$-dimensional submanifold of $U$ whose mean curvature vector field $H$ satisfies $H E_q \leq 0$. Let $u : \Sigma \to \R$ be the restriction of $f_{K,q}$ to $\Sigma$. As in equation (\ref{eq:laplacian-of-restriction}), \begin{align*}
\Delta u \,(p) &= \sum_{i = 1}^{k} (\overline{\nabla}^2 f_{K,q})_p(e_i,e_i) - k \, g(H_p, {(\overline{\nabla} f_{K,q})}_p) \\
&\geq k (1 - K f_{K,q}(p)) - k\,g(H_p, {(\overline{\nabla} f_{K,q})}_p).
\end{align*}

Simple computation yields 
$$
\overline{\nabla} f_{K,q}
=\begin{cases}
\overline{\nabla} E_q/2, & K=0,\\
\frac{\sin \sqrt{KE_q}}{2\sqrt{KE_q}} \, \overline{\nabla} E_q , & K\ne 0,
\end{cases}
$$ 
where the argument of $\sin$ can be imaginary here. The function $\sin \sqrt{KE_q}/(2\sqrt{KE_q})$ is non-negative as long as $KE_q \leq \pi^2$. Thus, $g(H_p, {(\overline{\nabla} f_{K,q})}_p) \leq 0 $ on $U$ since $g(H,\overline{\nabla}E_q) = HE_q \leq 0$. 

Since $(1 - K f_{K,q})|_{\Sigma}>0$,
we conclude that $u$ is subharmonic and satisfies the differential inequality \begin{equation} \label{eq:udiffineq}
\Delta u \geq k( 1 - Ku ) > 0.
\end{equation} 
By (\ref{eq:modE}), 
$u^* = \sup_{\Sigma} u < +\infty$. Since $\Sigma$ is stochastically complete, we can apply the weak maximum principle to obtain a sequence of points $p_n \in \Sigma$ such that $$
u(p_n) > u^* - \frac{1}{n} \quad \text{and} \quad \Delta u(p_n) < \frac{1}{n}.
$$ 
Evaluating (\ref{eq:udiffineq}) on $p_n$ and taking $n \to \infty$, we obtain 
$1 - K u^* = \cos \sqrt{KE^*} = 0$
, where $E^* = \lim_{n \to \infty} E_q(p_n)$. 


If $K = 0$, this is impossible. 
If $K > 0$
and $\sup_{\Sigma} E_q < \pi^2/4K$, then 
$KE^* < \pi^2/4 $ and $\cos\sqrt{KE^*} > 0$, a contradiction. Similarly, if $K < 0$ and $\inf_{\Sigma} E_q > \pi^2/4K$, 
then $KE^* < \pi^2/4$ and $\cos\sqrt{KE^*} > 0$, a contradiction.
Hence (ii) and (i).

Finally, suppose 
$K\ne 0$ and 
$U\subset \tilde{U}$,  where $\tilde{U}$ is the diffeomorphic image under $\exp_q$ of a  star-shaped region in $T_qM$ about $O$, and $\,E_q:\tilde{U}\to\R\,$  satisfies $\,E_q < \pi^2/K\,$ if $\,K> 0$ and 
$\,E_q > \pi^2/K\,$if $\,K<0$. 

Suppose $\Sigma$ is a stochastically complete spacelike submanifold in $\tilde{U}$. Choose a sequence $p_n \in \Sigma$ as above and let $E^* = \lim_{n \to \infty} E_q(p_n)$. 
By the above calculation, we know that $KE^* \geq \pi^2/4$. If $K > 0$, then $E^* \geq \pi^2/4K$ and if $K < 0$, $E^* \leq \pi^2/4K$. If $K>0$, then $E^* = \inf_{\Sigma} E_q$ and if $K < 0$, then $E^* = \sup_{\Sigma} E_q$. Thus, in either situation $\Sigma$ does not enter $U$.
Hence (iii).
\end{proof}

Note that for $K>0$, the bounds on $E_q$ in Theorem \ref{thm:trapped-SC} affect only spacelike geodesics, and for $K<0$, only timelike geodesics.

\begin{remark} 
Where a weakly future-trapped submanifold $\Sigma$ intersects the causal future of $q$,  the condition (\ref{eq:HE_q}), namely $H E_q \leq 0$, is immediate. Where $\Sigma$ enters the causal past of $q$,  (\ref{eq:HE_q}) implies $H=0$. At a point $p$ not causally related to $q$,  (\ref{eq:HE_q}) restricts $H$  to a subcone of 
the cone of future directed vectors at
 $p$\,:  either $H\ne 0$ lies in a closed half-cone of 
the cone of future directed vectors at
  $p$, or $H$ is null and future-pointing, or $H=0$. 

For example, in Minkowski space, consider points $ v\in \Sigma$ where $v$ is  spacelike. If  $v$ approaches $v_0\ne 0$ in the future null cone of the origin $0$, these half-cones approach  the causal future cone of $0$; if  $ v$ approaches  $ v_0\ne 0$ in the past null cone of  $0$, these half-cones approach the light ray through $v_0$. \end{remark}

\section{Conclusion}

We have demonstrated  a close connection between 
sectional curvature bounds  of the form $\mathcal{R}\le K\,$ and space-time convex and $\lambda$-convex functions ($\lambda > 0$).  
We have constructed new $\lambda$-convex functions.
 We have used these functions to find   new  domains
that do not support trapped submanifolds.  
 
Our goal has been  to explain  some viewpoints and tools, rather than to give an exhaustive treatment. We plan a more systematic  treatment of results in future.
              
Note that the $\lambda$-convex functions considered here are based on  signed energy functions.  It would be interesting to identify other classes of $\lambda$-convex functions to which Theorem \ref{thm:miots local} can be applied.

\end{document}